\documentclass[a4paper,11pt]{article}

\usepackage{amsmath}
\usepackage{amsthm}
\usepackage{amsfonts}
\usepackage{amssymb}
\usepackage{color}
\usepackage{graphicx}
\usepackage{xspace}
\usepackage{bbm}
\usepackage{enumitem}

\newcommand\amax{A_{\max}}
\newcommand\losing{\cL}
\renewcommand\th{^\textrm{th}}
\newcommand\PI{PI\xspace}
\newcommand\PII{PII\xspace}
\newcommand\one{{\mathbbm 1}}

\newcommand\cL{{\mathcal L}}
\newcommand\Z{{\mathbb Z}}

\newcommand\cF{{\mathcal F}}

\newcommand\cW{{\mathcal W}}

\newcommand\cT{{\mathcal T}}

\newcommand\cH{{\mathcal H}}

\newcommand\eps{\varepsilon}

\title{Transitive Avoidance Games}
\author{J.~Robert Johnson\footnote{School of Mathematical Sciences, Queen Mary,
University of London, London E1 4NS, England} \footnote{\tt r.johnson@qmul.ac.uk}\and Imre Leader\footnote{Department
of Pure Mathematics and Mathematical Statistics,
Centre for Mathematical Sciences,
Wilberforce Road,
Cambridge CB3 0WB,
England.} \footnote{{\tt I.Leader@dpmms.cam.ac.uk}}
 \and Mark Walters$^*$\footnote{\tt m.walters@qmul.ac.uk}}

\begin{document}

\newtheorem{theorem}{Theorem}
\newtheorem{lemma}[theorem]{Lemma}
\newtheorem{proposition}[theorem]{Proposition}
\newtheorem{corollary}[theorem]{Corollary}
\newtheorem{question}{Question}
\newtheorem{problem}[question]{Problem}
\newtheorem*{quotedresult}{Lemma}
\newtheorem*{conjecture}{Conjecture}
\newtheorem*{defn}{Definition}
\theoremstyle{remark}
\newtheorem*{remark}{Remark}

\maketitle

\begin{abstract}
  Positional games are a well-studied class of combinatorial game. In
  their usual form, two players take turns to play moves in a set
  (`the board'), and certain subsets are designated as `winning': the
  first person to occupy such a set wins the game. For these games, it
  is well known that (with correct play) the game cannot be a
  second-player win.

  In the avoidance (or mis\`{e}re) form, the first person to occupy
  such a set \emph{loses} the game. Here it would be natural to expect that
  the game cannot be a first-player win, at least if the game is
  transitive, meaning that all points of the board look the same.  Our
  main result is that, contrary to this expectation, there are
  transitive games that are first-player wins, for all board sizes
  which are not prime or a power of~2. 

  Further, we show that such games can have additional properties such
  as stronger transitivity conditions, fast winning times, and `small'
  winning sets.
\end{abstract}

\section{Introduction}
Many natural combinatorial games may be viewed as `achievement games'
as follows. We have a finite set (the \emph{board}) and some of its
subsets are designated as special and called \emph{lines}. Two players take
it in turn to claim a (previously unclaimed) point of the board; the
first player to complete a line is declared the winner.  If all points
have been claimed but neither player has completed a line then the
game is considered a draw.

In this paper we consider the avoidance (or mis\`{e}re) variant -- the
game is played exactly as above except the first player to complete a
line \emph{loses}. Simmons~\cite{simmonssim} described the first
example of such a game that we are aware of (the game of Sim), and
Harary~\cite{MR671910} introduced the more general mathematical
framework. Since then such games have been considered by many authors
-- see, e.g., Beck~\cite{MR2402857,MR1898254} and
Slany~\cite{DBLP:journals/eccc/ECCC-TR99-047}. For more general
mis\`ere games see, e.g., Conway~\cite{MR1808891}, and
Albert and Nowakowski~\cite{MR2588534}. There has also been a substantial amount
of work on Avoider/Enforcer games, which are the `Maker/Breaker' analogue
for avoidance games -- see, e.g., Lu~\cite{MR1138598}, Beck~\cite{MR2402857},
and Hefetz, Krivelevich and Szab\'{o}~\cite{MR2333136}.

For achievement games, a simple strategy stealing argument shows that
the game is either a draw or a first player win (with perfect
play). However, for avoidance games the situation is not clear. Indeed
Beck~\cite{MR1898254} states ``The general open
problem is to find the avoidance version of the strategy stealing
argument.''

At first glance it looks as though the first player has a disadvantage
since he has `more' points than the second player. Of course in
general this is not true: indeed take any game that is a second player
win and add a single point not in any line. The first player picks
this new point on his first turn and thus reduces the new game to the
old game with Player~II (the second player) playing first.

However, this is a rather trivial example: we have artificially given
Player~I an advantage by giving him a special `safe' point he can
pick. Thus, it is natural to insist that the game is transitive: that
is, that the automorphism group of the game acts transitively on the
board's points. (The automorphism group is the group of all
permutations of the board $X$ that preserve the family $\losing$ of
lines.)  Informally, `all points look the same.'

One might guess that a transitive avoidance game must be a second
player win: what advantage can there be to going first when all points
are the same? Indeed, this intuition is correct when all the lines
have size 2. In general, however, it turns out that there may be
advantages to playing first.

A priori, there are two reasons why it seems to be harder to find a
game that is a Player I win when $n$ (the board size) is odd: first,
Player I has an extra move and, secondly, Player II has the final
choice of move (as Player I's last move is forced). However, whilst we
shall give examples of first player wins for both $n$ even and $n$
odd, it turns out that the simplest examples of such a game occur in
the case $n$ odd (see Section~\ref{s:odd}).

Having found examples of Player~I wins we turn to the central question
we address in this paper: namely for which $n$ does there exist a
transitive avoidance game that is a Player~I win.  Our main result
shows that, for most $n$, such games do exist.  \color{black}
\begin{theorem}\label{t:all-n}
  Suppose that $n$ is neither a power of 2 nor a prime. Then there is
  a transitive avoidance game on $n$ points that is a Player I
  win.
\end{theorem}
The case when $n$ is divisible by a large power of 2 is much harder
than the case when $n$ is odd, or $n$ is equal to $2 \mod 4$. It relies on a
careful study of subsets of $\Z_{2^m}$ (the integers modulo $2^m$) under rotation. This may be of
independent interest.

There are two cases not covered by Theorem~\ref{t:all-n}: powers of
2, and primes. For powers of 2 we have the following result.
\begin{theorem}
  Suppose that $n$ is a power of 2. Then no transitive avoidance game
  on $n$ points is a Player I win.
\end{theorem}
It turns out that this is a fairly simple consequence of a well known
group theoretic result.

For the remaining case of $n$ prime there are examples for which there
are Player I wins (e.g., 11 and 13), but we do not know what happens
for any prime larger than 13.

In many of our constructions Player II does not lose until his last
move. It is easy to modify these examples so that Player I can force
Player II to lose before time $(1-\eps)n$.  However, we also give
examples of games where Player~II loses in time $o(n)$. Interestingly
these even include examples where the losing lines are all of bounded
size (size 3 in fact).

We also consider avoidance games where on each turn each player is
allowed to pick more than one point (this is usually called the
\emph{plus} version of the game; see
Slany~\cite{DBLP:journals/eccc/ECCC-TR99-047}). One might think that
picking more than one point would never help a player, but it turns
out that in any transitive plus game Player II can use this extra
freedom to guarantee that he does not lose. A related phenomenom was
proved by Hefetz, Krivelevich, Stojakovi\'{c} and
Szab\'{o}~\cite{MR2557886} for a biased version of the
Avoider/Enforcer game where moving to the plus version of the game
simplified the behaviour substantially.

Our argument applies to some previously studied games. For example, it
gives a simple strategy stealing proof of the previously unknown fact
that the plus version of the Ramsey Avoidance Game (see
Section~\ref{s:strategy-stealing}) is never a Player I win.

\medskip In Section~\ref{s:strategy-stealing} we give two
situations where strategy stealing arguments do work: when there is a
line of size 2, and for the plus version of the games.  Then in
Section~\ref{s:odd} we give a simple construction of Player~I wins for
all composite odd~$n$.

Section~\ref{s:even} is the heart of the paper. In it we show that
there are Player~I wins for all even sizes, except powers of 2. 

In Section~\ref{s:simple} we discuss some natural variants of the game
such as lines with bounded size, games which are `more' transitive,
and games where Player I can win quickly.  We show that our
constructions so far (or simple modifications of them) are able to
provide first player wins in these cases too.

Then, in Section~\ref{s:torus} we give a natural game which satisfies
all of these stronger properties simultaneously -- it is a fast Player
I win, the lines have bounded size, and it is more transitive (more
precisely it is edge/line transitive -- see the section for the formal
definition).

We conclude with a discussion of some open problems.

\medskip We will use the following notation for games throughout the
paper. We will denote the board by $X$, the board size $|X|$ by $n$,
and the set of losing lines by $\cL$. We abbreviate Player~I and
Player~II to \PI and \PII respectively.  

\section{Games for which strategy stealing works}\label{s:strategy-stealing}
\subsection{A line of size 2}

If the family of losing lines $\losing$ has any set of size 2 then
there is a strategy stealing argument. Note that this includes the
case when all lines have size 2: i.e., where $\losing$ corresponds
to the edge set of some vertex transitive graph.
\begin{theorem}\label{t:graph}\label{t:size2=p2-win}
  A transitive avoidance game with any line of size 2 is not a \PI win.
\end{theorem}
\begin{proof}
Suppose (for contradiction) that the game is a \PI win. 

Suppose that \PI's first move is $x$. Since the game is transitive,
every point is in a line of size 2 so we can choose $y$ such
that $\{x,y\}\in \losing$.

Using the transitivity of the game again we see that there is a
winning strategy, $\Phi$ say, for \PI with first move $y$.

\PII now plays strategy $\Phi$ ignoring the fact that \PI has already
claimed~$x$. This could only go wrong if $\Phi$ tells \PII to
play~$x$. However, since $\{x,y\}$ is a line, and thus \PII would lose
anyway if he played $x$, this cannot happen.
\end{proof}
Note that we made crucial use of the existence of a line of size 2
in the `\PII would lose anyway' statement at the end of the proof.

We remark that a similar argument shows that in any transitive
avoidance game any strategy of \PI that promises never to play a
particular point is not winning. Thus, winning strategies for \PI must
be global; i.e., they must examine the entire board.

\subsection{The plus version}
As described in the introduction, the plus variant of an avoidance game is
the same as the avoidance game except that, on each move, a player may
choose to pick as many points as he likes rather than just one; in
other words, on each move, each player chooses a (non-empty) set of
points.

\begin{theorem}\label{t:plus}
  The plus version of a transitive avoidance game is not a \PI win.
\end{theorem}
\begin{proof}Since the players can pick an arbitrary number of points
  we cannot tell which player's move it is just by looking at the
  current position, and we will make use of this fact in our
  proof. Thus, we introduce some notation for the position. A position
  is a pair of disjoint sets $(S,T)$ where $S$ denotes the set of 
  all points picked so far by the player whose turn it is, and $T$ the set
  of all points picked by the other player. From this position we will
  name the player whose move it is the \emph{current} player.

  Let $G$ be the automorphism group of the game.

  Suppose for the sake of a contradiction that \PI has a winning
  strategy and let $S$ be \PI's first (set) move in this
  strategy. After this move the position is $(\emptyset,S)$.  Since
  \PI is playing a winning strategy we know that position $(S,T)$ is a
  current player win for any non-empty set $T$. Thus, for any $g\in G$
  and any non-empty set $T$ disjoint from $g(S)$, the position
  $(g(S),T)$ is also a current player win.

  To obtain a contradiction we now show that the game is not in fact
  is a \PI win. Suppose that \PI plays some set $U$. If $U$ is the
  whole board then the game is over and \PII had definitely not
  lost. Thus, we may assume that \PI does not pick all the points.

  We have to give a strategy for \PII. Pick $x\not \in U$ and $g\in G$
  which maps some $s\in S$ to $x$. Let $S'=g(S)$.  \PII plays
  $S'\setminus U$, which contains $x$ so is non-empty.

  Now, suppose that \PI plays $W$. Then, it is  \PII's turn and the position
  is $(S'\setminus U,U\cup W)$. Since $S'\setminus U\subset S'$ and 
  \[(S'\setminus U)\cup (U\cup W)=S'\cup U\cup W=S'\cup ((U\setminus
  S')\cup W),\] we see that $(S'\setminus U,U\cup W)$ is no worse for
  the current player than $(S',(U\setminus S')\cup W)$. As noted
  above, this latter position is a current player win so $(S'\setminus
  U,U\cup W)$ is also a current player win; i.e., \PII wins as
  claimed.
\end{proof}
Some plus version avoidance games have been studied previously: in
particular, the Ramsey Avoidance Game $RAG(N,s)$ (see,
e.g.,~\cite{DBLP:journals/eccc/ECCC-TR99-047,MR1898254}).  This game
is played with board the edge set $E=E(K_N)$ and lines are the
$\binom{s}{2}$-subsets of $E$ corresponding to $K_s$ subgraphs for
some fixed $s$.

\begin{corollary}
  The plus version of Ramsey Avoidance Game $RAG(N,s)$ is not a \PI
  win. In particular, if $N\ge R(s,s)$ (where $R(s,s)$ denotes the
  Ramsey number) then it is a \PII win.
\end{corollary}
\begin{proof}
  The game is obviously transitive so Theorem~\ref{t:plus} implies that it
  is not a \PI win. In the case $N\ge R(s,s)$ the game cannot be a
  draw -- when the board is full (at least) one player must have a
  $K_s$ -- and so must be a \PII win.
\end{proof}

\section{First player wins for odd board sizes}\label{s:odd}
In this section we show that there are first player wins for all odd
composite board sizes. This is in contrast to the special cases 
discussed in the previous section.
\begin{theorem}\label{t:n-composite-odd}
  Suppose that $n=pq$ is odd. Then there is a transitive avoidance game on $[n]$ that is a \PI win.
\end{theorem}
\begin{proof}
  First we define the game. View the board, $[n]$, as $q$ sets
  $A_1,A_2,\dots,A_q$ of $p$ points. Let $p'=(p+1)/2$ and
  $q'=(q+1)/2$.  We call each $A_i$ a \emph{bucket}.

  Let $\cW$ be all subsets of $[n]$ of size $p'q'$ consisting of $p'$
  from each of $q'$ buckets.  Define the lines
  $\losing=[n]^{(p'q')}\setminus \cW$.

  We claim that the avoidance game on $(n,\losing)$ is transitive and
  a \PI win. The transitivity is trivial since we can permute the
  bins, and permute the points in any bin. Thus we just need to prove
  that it is a \PI win.

  Since, every pair of sets in $\cW$ meet, we see that if \PI can form
  a set in~$\cW$ with his first $p'q'$ points then \PII's first $p'q'$
  points must form a set not in $\cW$, i.e., these points must form a
  line in $\losing$.  Thus, it suffices to show that \PI can guarantee
  to form a set in $\cW$ with his first $p'q'$ points.

  We call a bucket \emph{active} if \PI has played at least one point
  in it but not yet $p'$ points in it. We say it is \emph{full} if he
  has played $p'$ points in it.  \PI's strategy is to play according
  to first of the following rules that applies.
  \begin{enumerate}
    \item If \PII has just played in an active bucket then \PI plays in the same bucket.
    \item If less than $q'$ buckets are either full or active
    then \PI plays in an empty bucket.
    \item \PI plays in any active bucket.
  \end{enumerate}
  Rule 1 implies that, after his turn, \PI always has strictly more
  points than \PII in any active bucket. Rule 2 implies that, after
  \PI's turn, strictly more than half the non-empty buckets are active
  or full. Thus after $p'q'$ moves \PI has exactly $p'$ points in
  exactly $q'$ buckets.
\end{proof}

\section{First player wins for even board sizes}\label{s:even}
\subsection{Isbell Families}
We saw in the introduction that one advantage for the first player is
that he goes first and can pick a `special' point, but obviously this
is not possible for a transitive game. A second possible advantage for
the first player occurs if the board has an even number of points: the
second player has no choice on his last turn, whereas the first player
does always have a choice. We use this `forced move' to construct
examples of games with even size boards that are first player wins.

It turns out that our avoidance game is closely related to an existing
game idea, namely that of a \emph{fair game}, from a very different
context (see Isbell~\cite{MR0093739} and~\cite{MR0118541}). To avoid confusion with other notions of
fair game we will call the families involved \emph{Isbell families}.
\begin{defn}
  An \emph{Isbell family} on a set $[n]$ is a family $\cF$ of subsets of
  $[n]$ such that $\cF$ is an up-set containing exactly one of $X$ and
  $[n]\setminus X$ for each set $X$, and having a transitive automorphism group.
\end{defn}
We remark that an Isbell family must be \emph{intersecting}, that is
any two sets in the family meet.

\begin{proposition}\label{p:fair-games}
  Suppose that $n$ is even and that there exists an Isbell family
  on~$[n]$.  Then there is a transitive avoidance game with board
  $[n]$ that is a first player win.
\end{proposition}
\begin{proof}
  Let $\cF$ be the sets in the Isbell family. We define the lines
  $\losing=\cF\cap [n]^{(n/2)}$. We show that this avoidance game is a \PI
  win. Obviously, this game is transitive.

  Consider the achievement game on board $[n]$ with winning lines
  $\cW=[n]^{(n/2)}\setminus \losing$. By the definition of an Isbell
  family every $n/2$ sized set is either in $\cW$ or its complement is
  in $\cW$. Hence, a draw is impossible in this achievement
  game. Thus, by the standard strategy stealing result, this
  achievement game must be a \PI win.

  \PI follows exactly the same strategy in the avoidance game. At the
  end his $n/2$ points form a set in $\cW$ so not in $\losing$ (so he
  has not lost) and \PII's $n/2$ points are the complement of \PI's
  set, so must form a set in $\losing$ and \PII has lost.
\end{proof}
It is not known for exactly which $n$ Isbell families exist but most
cases are known. In particular, Isbell~\cite{MR0093739} showed that
they do exist for $n=2b$ with $b>1$ odd; and Cameron, Frankl and
Kantor~\cite{MR988508} showed they do exist for $n=4b$ with $b>3$ and
odd. Thus, for all these cases we have transitive avoidance games
which are first player wins.  However, Cameron, Frankl and
Kantor~\cite{MR988508} also showed that Isbell families do not exist
for $n=2^a$, or for $n=3\times 2^a$ for $a\ge 2$, so we cannot use them to
prove all the remaining cases of Theorem~\ref{t:all-n}.

For concreteness let us describe the Isbell family, and
thus the \PI win avoidance game, on $6$ points.  We think of the six
points as being arranged in a  grid of two rows and three columns. The Isbell family
$\cF$ is the up-set generated by the family of all 3-sets that either
contain one point from each pair and an even number of points in the top
row, or contain both points in one pair and one point in the next pair
cyclicly. It is easy to check that this family is transitive and that every
set is either in $\cF$ or its complement is in $\cF$. Thus, $\cF$ is
indeed an Isbell family, and the avoidance game with lines
$[6]^{(3)}\cap \cF$ (i.e., the lines are exactly the generating sets
described above) is a \PI win. Indeed, this is easy to verify by hand.

In contrast, there do not exist first player wins when $n$ is a power
of 2. In order to prove this we start with a simple lemma.
\begin{lemma}\label{l:fpfi}
  Let $(n,\losing)$ be a transitive avoidance game and suppose that
  its automorphism group contains a fixed-point-free involution. Then
  the game is not a \PI win.
\end{lemma}
\begin{proof}
  Let $G$ be the game's automorphism group and let $g\in G$ be a
  fixed-point-free involution.  Then $g$ partitions the board into
  pairs.  \PII's strategy is to play the point paired with \PI's
  previous move; i.e., if \PI plays $x$, then \PII plays $g(x)$.
  
  Suppose that, after \PII's move, \PII has played all the points in
  $T$. Then \PI must have played all the points in $g(T)$. Thus, if
  \PII's move was losing -- that is, if $T$ contains a losing set -- then
  \PI must have already lost.
\end{proof}
The following theorem is an immediate consequence.
\begin{theorem}
  A transitive avoidance game on a set of size~$2^a$ for
  some~$a$ is not a \PI win.
\end{theorem}
\begin{proof}
  Let $G$ be the automorphism group of the game.  Then $G$ acts
  transitively on the board $X$ which has size $2^a$. By a standard
  result (see e.g.,~\cite{MR988508}) from group theory, it follows
  that $g$ has a fixed-point-free involution. Thus, by
  Lemma~\ref{l:fpfi} the game is not a \PI win.
\end{proof} 

We would like to extend Proposition~\ref{p:fair-games} to show that
there is a transitive avoidance game that is a \PI win for all even
sizes except powers of~2.  In the proof of
Proposition~\ref{p:fair-games} we relied on the existence of Isbell
families.  These were useful because they gave a transitive
intersecting family consisting of half of the $\frac{n}{2}$-subsets.
Since there do not exist Isbell families for sizes $3\times 2^a$, we
cannot use the same technique. Indeed, since the union of any such
intersecting family and the family of all subsets of size strictly
greater than $n/2$ is an Isbell family, such intersecting families
only exist when Isbell families exist.

Instead, we look for a smaller intersecting family $\cF$ for which \PI
can win the achievement game. Since $\cF$ is not a maximal
intersecting family it is possible that neither player forms a set
from $\cF$ (i.e., a draw is possible in the achievement game), so we
cannot use a strategy stealing argument to show that \PI has a winning
strategy for the achievement game. Thus, we need to define a winning
strategy explicitly.

\subsection{Special case: $n=2b$ with $b$ odd}
As our proof in the general case is rather involved, we start by
illustrating it in the simple case of $n=2b$ with $b$ odd; this case
is covered by Isbell's results on the existence of Isbell Families,
but we will give an example with a much smaller intersecting family
for all $n\ge 10$. (For $n=6$ the family is exactly the Isbell family
on six points described earlier.) We remark that we give a proof that
generalises to the full case, rather than the simplest proof for this
special case.

\begin{proposition}\label{p:2b}
  For any odd $b\ge3$ there is a transitive avoidance game of size~$2b$.
\end{proposition}
\begin{proof}
  Let $n=2b$ and $b'=(b-1)/2$. We think of $[n]$ as $\Z_b\times \Z_2$:
  i.e., as~$b$ pairs.  We start by defining some \emph{winning} sets
  $\cW$. They are all of size $b$, and are
\begin{enumerate}
  \item\label{e:2b-1} all sets which contain exactly one point from each pair and an
  odd number of which have a 1 in the $\Z_2$-coordinate,
  \item\label{e:2b-2} all sets with both elements of exactly one pair, that have the
  (unique) pair with neither element at most $b'$ later
  cyclicly.
\end{enumerate}
Obviously the complement of any set in this family is not in the
family and thus we see that this family is intersecting.  We define
the board of our avoidance game to be~$[n]$ and the lines~$\losing$
to be the complements of the sets in~$\cW$.  We claim that this
avoidance game is transitive and a \PI win.

To see that the game is transitive, it is enough to observe that its
automorphism group contains the following elements
\begin{itemize}
  \item cycle the $b$ pairs, 
  \item in two pairs swap the elements (i.e., swap~$(x,0)$
  and~$(x,1)$, and swap~$(y,0)$ and~$(y,1)$ for some $x$ and $y$).
\end{itemize}

Thus, to complete the proof, we just need to show that the game is a
\PI win. Indeed, if we can show that \PI can guarantee to make a set
in $\cW$, then we know that \PII will finish with a set in $\losing$,
and so will lose.

We will denote positions in the game as ordered pairs $(A,B)$ of
subsets of the board where this means \PI has played the points in $A$
and \PII has played the points in $B$. For a set $A$ containing at
most one of each pair, we write $\bar A$ for the `opposite' points:
that is the points in the other half of each pair that meets $A$.

First, define a position to be a \emph{direct win} if it is of the
form $(A\cup \{(x,y)\},\bar A\cup \{(z,w)\})$ for some $A$ and $1\le
z-x\le b'$, and observe that \PI has a simple winning strategy from
any such position. Indeed, he plays $(x,y+1)$ and for the rest of the
game makes sure he gets one of every other pair set except he never
plays the point $(z,w+1)$ opposite to $(z,w)$. This means that after
\PI's last turn he has one of every pair, except he has both of the
$x\th$ pair and neither of the $z\th$ pair. The constraint on $z$
means that this is a winning set under Condition~\ref{e:2b-2}
above. 

Our strategy for \PI is as follows. Unless the position is a direct
win, he will ensure that the position after his turn is of the form
$(A\cup \{(x,y)\},\bar A)$ for some set $A$ and point $(x,y)$.

Now from the position $(A\cup \{(x,y)\},\bar A)$ suppose that \PII plays a
point $(z,w)$. We split into several cases.
\begin{enumerate}
  \item\label{e:2b-win-1} If $1\le z-x\le b'$ then the position is a
  direct win and he wins by following the simple strategy given above.

  \item If $(z,w)$ is any point other than $(x,y+1)$ (i.e., unless
  \PII plays the opposite point to $(x,y)$) then \PI plays the point
  $(z,w+1)$ (i.e., opposite where \PII played). Obviously this is
  possible and keeps the position being of the form above, so \PI
  follows the same strategy from the new position.

  \item The last case is if \PII plays the point $(x,y+1)$, i.e., the
  opposite point to $(x,y)$. Thus, before \PI plays, the position is
  $(B,\bar B)$ where $B=A\cup\{(x,y)\}$: i.e., every pair is either
  full or empty.  Let $x'$ be the first empty pair. If $x'<b'$ then \PI
  just plays $(x',0)$ making the position of the required form and
  continues as above.

  (Note this case implies that \PI plays $(0,0)$ on his first go.)

  There is one remaining case: $x'\ge b'$. We deal with this below.
\end{enumerate}
To have reached this position all pairs $[0,b')$ must already be
filled. There may be some points in pairs in the interval $(b',b)$ but
the ${b'}\th$ pair must be empty as \PI did not play in it, and if \PII
had played in it then he would have lost under Case~\ref{e:2b-win-1}
above (as for the whole game so far the `extra point' in the position
was in a pair in the interval $[0,b')$).

For the remainder of the game \PI is going to play in a certain
fashion filling up the empty pairs $[b',b)$ in turn, except if the
position is a direct win, in which case \PI follows that strategy.
Since, when \PI plays in the empty pair $x$ all pairs $[0,x)$ have
already been filled, unless \PII plays the other point in pair $x$
then the position will be a direct win and \PI wins. Thus, we may
assume that \PII plays the other point in pair~$x$.

Hence, \PI gets to pick one point from each of the remaining empty
pairs, and \PII has to pick the other from each of them. In
particular, when \PI picks his point from the last empty pair he can
choose the correct element to ensure the correct parity under
Condition~\ref{e:2b-1}. Thus, \PI finishes with a set in $\cW$ as
required.
\end{proof}

\subsection{General even case}
Having seen this special case, we extend these ideas to all board
sizes of the form $n=2^ab$ with $b$ odd and greater than 1 (i.e., all
even board sizes except powers of 2). As above, we view this as $b$
copies of $2^a$; we call each copy of $2^a$ a \emph{bin}. In the
previous example there were very few possibilities for what happened
in one copy of $2^a$ (i.e. in a pair) but in the general case there
will be a lot more.  This makes the proof substantially more
difficult.

We need some definitions. Let $m=2^a$ and let $m'=m/4$.
\begin{defn}For $x\in[m]$ we define the \emph{opposite point} to $x$
  to be be the point $x+m/2$, and we call the pair $\{x,x+m/2\}$ an
  \emph{opposite pair}. We say a set $A\subset [m]$ is a \emph{cyclic
    pair set} if it contains exactly one of each opposite pair. We say
  it is a \emph{partial cyclic pair set} if it contains at most one of
  each opposite pair and is not empty. In a partial pair set we call
  any point where neither it nor its opposite point are in the set a
  \emph{free} point.
\end{defn}

The proof relies on a careful examination of the lexicographic order
on a set and its rotations and, in particular, its `maximum' rotate
defined as follows. 
\begin{defn}
  The lexicographic order on subsets of $[m]$ is defined as follows:
  $A\le B$ if the first point in the symmetric difference $A\Delta B$
  is in $A$.

  For any $r>1$, an \emph{$r$-maximal point} of a cyclic pair set
  $A$ is a point $x$ where the intersection of $A$ with the interval
  $\{x,x+1,x+2,\dots,x+r-1\}$ 
is maximal in lexicographic order over all the sets
  formed by intersecting  $A$ with an interval of length $r$ of
  $[m]$. We say $x$ is \emph{maximal} if it is an $m$-maximal point.
  We define \emph{$r$-minimal} and \emph{minimal} similarly.
\end{defn}
\begin{lemma}\label{l:unique-max}
  Any partial cyclic pair set in $[m]$ has a unique maximal point.
\end{lemma}
\begin{proof}
  Since the set contains exactly one of some opposite pair the set is
  not fixed by cycling by $m/2$, the order of the stabiliser is not
  divisible by $2$. Since the set has size $m=2^a$ this means the
  stabiliser has order 1; in particular, no two cyclic shifts of $A$
  are the same so there must be a unique maximal point.
\end{proof}

In the case when $a=1$ discussed above, i.e., for board sizes $2b$, it
was important that when \PI came to decide what happened in the final
empty pair he already knew what would happen in all the remaining
pairs. In that case that was trivial: they were already completely
filled. In the general case, when \PI decides what happens in the
final empty bin the remaining bins are non-empty but this does not
mean they are full. However, the following key lemma shows that \PI
has some control over what happens in these later bins.
\begin{lemma}\label{l:key-lemma}
  Suppose that $m$ and $m'$ are as above, and that $A$ is a partial
  pair set in $\Z_m$. Then there exist values $s,t$ and $z_1,z_2$ such
  that
  \begin{itemize}
    \item if all free points in $[z_1,z_1+m')$ are placed in $A$ then
    regardless of which of the remaining free points
    are placed in $A$ the maximum lies in $[t-s,t]$
    \item if all free points in $[z_2,z_2+m')$ are placed in $A$ then
    regardless of which of the remaining free points are placed in $A$ the
    maximum lies in $[t,t+2m'-s)$.
  \end{itemize}
\end{lemma}
We postpone the proof of this technical lemma to later. First, we show
that we can deduce the existence of transitive avoidance games of all
even sizes (except powers of 2) that are \PI wins from it. The
deduction is similar to the proof of Proposition~\ref{p:2b}.
\begin{theorem}
  Suppose that $n=2^ab$ for some odd $b>1$. Then there is a transitive
  avoidance game on $n$ points which is a Player~1 win.
\end{theorem}
\begin{proof}
  Let $m=2^a$, $m'=m/4$ and $b'=(b-1)/2$.  First, we define the winning
  sets $\cW$. These will all have size $n/2$ and the lines in our avoidance
  game will be the family of their complements. We view $n$ as being
  $\Z_b\times Z_m$.  For any point $(x,y)$ we define its opposite
  point to be the opposite point in the same bin: i.e., $(x,y+2m')$.

  The winning sets $\cW$ are all the sets
  \begin{enumerate}
    \item that contain exactly one of each opposite pair and the sum
    over all bins of the maximal points mod $m$ lies in the interval
    $[0,m/2)$.
    \item \label{condition-2}that contain both of exactly one opposite
    pair, say the pair $\{y,y+2m'\}$ in bin $j$, and the unique empty
    opposite pair is either
    \begin{enumerate}
      \item \label{condition-2a} in one of the bins between $j+1$ and $j+b'$,
      \item \label{condition-2b} in bin $j$ and of the form
      $\{z,z+2m'\}$ for some $y+1\le z\le y+m'-1$ (i.e., the empty pair is
      between 1 and $m'-1$ after the full pair in the same bin)
    \end{enumerate}
  \end{enumerate}
  This is an intersecting family; indeed, since all the sets have size
  $n/2$, we just need to check that the complement of any set in the
  family is not in the family. In the first case, this follows since
  the maximal point of the complement of a set containing exactly one
  of each pair is the maximal point of the set plus $m/2$, so the sum
  of the maximal points changes by $bm/2\equiv m/2$ modulo $m$. In the
  second case it is trivial.

  Also this family is transitive: indeed, the automorphism group
  contains the elements 
  \begin{itemize}
    \item cycle the $b$ bins
    \item rotate each bin $i$ by an amount $r_i$ with $\sum_i r_i=0$.
  \end{itemize}

  As in Proposition~\ref{p:2b} we will denote a position in the
  game by an ordered pair $(A,B)$ of subsets of the board, where $A$
  denotes the points played by \PI, and $B$ the points played by \PII.
  For any set $A$ containing at most one of each opposite pair we
  write $\bar A$ for the set of points opposite to $A$.  

  Also as in Proposition~\ref{p:2b}, some positions have simple direct
  wins. There are two types. A position of the form $(A\cup
  \{(x,y)\},\bar A\cup\{(z,w)\})$ is a \emph{direct win of type-1} if
  $1\le z-x\le b'$. It is a \emph{direct win of type-2} if $z=x$ and
  $0<w-y<m'$ or $2n'< w-y<3m'$. Let us see that both of these are
  indeed winning positions for \PI.

  If the position is a direct win of type-1 then \PI plays $(x,y+2m')$
  and for the rest of the game makes sure he gets one of every other
  pair set except he never plays the point $(z,w+2m')$ opposite to
  $(z,w)$. This means that after \PI's last turn he has one of every
  pair except he has both of a pair in the $x\th$ bin and neither of a
  pair in the $z\th$ bin. The constraint on~$z$ means that this is a
  winning set under Condition~\ref{condition-2a} above.

  If the position is a direct win of type-2 then \PI plays $x,y+2m'$
  and again makes sure he gets one of each opposite pair apart from he
  never plays $(z,w+2m')$.  This means that Maker finishes with a set
  that is winning under Condition~\ref{condition-2b} above.

  \PI's strategy is as follows. Unless the position is a direct win
  (of either type), he will make sure that the position after his turn
  is of the form $(A\cup \{(x,y)\},\bar A)$.  Now suppose that, from
  this position, \PII plays any point $(z,w)$.  We split into several
  cases
  \begin{enumerate}
    \item If $1\le z-x\le b'$ then the position is a direct win of
    type-1 and \PI wins.
    \item If $z=x$ and $0<w-y<m'$ or $2n'< w-y<3m'$ then the position
    is a direct win of type-2  and \PI wins.
    \item If $(z,w)$ is any point other than $(x,y+2m')$ (i.e., unless
    \PII plays the opposite point to $(x,y)$) then \PI plays the point
    $(z,w+2m')$ (i.e., opposite where \PII played). Obviously, this is
    possible and it keeps the position being of the form above, so \PI
    follows the same strategy from the new position.

    \item If \PII plays $(x,y+2m')$, the point opposite to
    $(x,y)$. Then, before \PI plays, the position is $(B,\bar B)$ where
    $B=A\cup \{(x,y)\}$.  Let $(x',y')$ be the empty point where $x'$
    is smallest and $y'$ is smallest amongst the empty points for that
    value of $x'$. If $x'<b'$ then \PI just plays $(x',y')$ making the
    position of the required form, and continues as above.

    (Note this case implies that \PI plays $(0,0)$ on his first go.)

    There is one remaining case: $x'\ge b'$. We deal with this below.
  \end{enumerate}
  
  There is one remaining case.  To have reached this position all
  points in bins $[0,b')$ must already be filled. There may be some
  points in bins $(b',b)$ but the bin $b'$ must be empty as \PI did
  not play in it and if \PII had played in it then he would have lost
  under Case 1 above (as for the whole game so far the `extra point'
  in the position was in bin $[0,b')$.

For the remainder of the game \PI is going to play in a certain
fashion filling up the bins $[b',b)$ in turn, except if a direct win
of either type occurs, in which case \PI follows the appropriate
winning strategy as described above.  Since, when \PI plays in a bin
$x$ all pairs $[0,x)$ have already been filled, unless \PII also plays
in the bin~$x$ the position will be a direct win and \PI wins. Thus,
we may assume that \PII plays in the same bin as \PI for the rest of
the game.

Observe that, when \PI first plays in any of the bins $[b',b)$, he can
pick a point $u$ and guarantee to play all the empty points in
$[u,u+m')$. Indeed, he starts by playing the first empty point after
$u$.  Then he follows his normal `pick the opposite point' strategy
(i.e., Case 3 of the strategy above) but whenever he gets a free
choice (i.e., Case 4) he picks the next free point in $[u,u+m')$. By
doing this \PI does get all the free points in $[u,u+m')$ since, if
\PII ever takes any of the points in $[u,u+m')$ or $[u+2m',u+3m')$,
then \PI wins directly under Case~2 above.

Let $r$ be the final empty bin. \PI plays
arbitrarily in the bins $[b',r)$, just ensuring that he gets one of
each pair. (For example he could guarantee to play all the points in
$[0,m')$ as described above.)

Now \PI is about to play in bin $r$. We define a key quantity which
\PI will try and control for the rest of the game; this quantity will
be defined for all positions where, for some $j>r$, bins $[0,j)$ are
full but bin $j$ is not.  In such a position, the maximal points $u_i$
for bins $[0,j)$ have all been determined and, moreover, all the
remaining bins are non-empty. Thus, for $j\le i<b$, there exist
$s_i,t_i$ as given by by Lemma~\ref{l:key-lemma}. We will think of the
number
\[G(j)=\sum_{i< j} u_i + \sum_{i\ge j} t_i
\]
as being the current `guess' at the sum of the maximum points. \PI's
strategy is to ensure that, for all $j>r$, this stays in the region
$[0,2m')$ as the bins fill up. We show that \PI can achieve this by
induction.

First, we show that he can play in bin $r$ to ensure that
$G(r+1)\in[0,2m')$. This is trivial: \PI picks a point
$u$ and, as above, guarantees to pick all the points in $[u,u+m')$.
It is easy to see that, however \PII plays, the maximum lies in
$[u-m',u]$. Thus, regardless of the game so far, by choosing $u$
correctly, \PI can ensure that $G(r+1)\in[0,2m')$.

Now suppose that $j>r$ and that $x=G(j)$.  Inductively, we know that
$x\in [0,2m')$, so one of $[x-s_j,x]$ and $[x,x+2m'-s_j)$ is a subset
of $[0,2m')$. Therefore, by the above observation, \PI can play all
the free points in $[z_1,z_1+m')$ or $[z_2,z_2+m')$ respectively,
where $z_1$, $z_2$ are as in Lemma~\ref{l:key-lemma}.  Then, however
\PII plays, $u_j\in[t_j-s_j,t_j]$ or $u_j\in[t_j,t_j+2m'-s_j)$ and
\[G(j+1)=G(j) +u_j-t_j\in [0,2m').
\]
Thus, $G(j+1)\in[0,2m')$ and the induction is complete.

When the process finishes the `guess' $G(b)$ equals the actual sum of
the maximum points and so this sum is in $[0,2m')$ and \PI wins.
\end{proof}

\subsection{Proof of  Lemma~\ref{l:key-lemma}}
In this section fix $m=2^a$ and let $m'=m/4$. We start with some
some notation.

Suppose that $A$ is a partial cyclic pair set in $\Z_m$. Then
$A|_{[x,x+r)}$ denotes the restriction of $A$ to the interval
$[x,x+r)$, rotated to be a subset of $[0,r)$.

We write $A+\one|_{x,x+r}$ for the union of $A$ and the set of all
free points in $[x,x+r)$. Note that this is not just $A\cup [x,x+r)$ as 
there may be points of $[x,x+r)$ that are opposite
to points in $A$ and these are not added; in particular,
$A+\one|_{x,x+r}$ is a partial cyclic pair set.

Finally we write $\amax$ for the set $A$ with all the free points
added. Note that $\amax$ is not a pair set (unless $A$ was already a
full pair set).

\medskip
Next we need some simple lemmas about maximal points.  We start with a
trivial observation that we use repeatedly.
\begin{lemma}
  Suppose that $A$ is a full pair set and that $x$ is $r$-maximal. Then
  $x+2m'$ is $r$-minimal.
\end{lemma}
\begin{proof}
  Immediate from the definition of a cyclic pair set.
\end{proof}

\begin{lemma}
  Let $A$ be any subset of $[m]$. Then any $r$-maximal point is $r'$-maximal for all $r'\le r$.
\end{lemma}
\begin{proof}
  This is trivial from properties of the lexicographic order.
\end{proof}
\begin{lemma}\label{l:max-r}
  Let $A$ be any subset of $[m]$ and suppose that $x$ is an $r$-maximal point and $x+r$ is an $r'$-maximal
  point. Then $x$ is an $(r+r')$-maximal point.
\end{lemma}
\begin{proof}
  This is trivial from properties of the lexicographic order.
\end{proof}
\begin{corollary}\label{c:max-r}
  Suppose that $A$ is a partial cyclic pair set and $x$ is a maximal
  point. Then the point $x-r$ is not an $r$-maximal point.
\end{corollary}
\begin{proof}
  By the previous lemma $x-r$ would be an $(r+m)$-maximal point so an
  $m$-maximal point (i.e., a maximal point) contradicting the
  uniqueness of the maximum point (Lemma~\ref{l:unique-max}).
\end{proof}
\noindent%
The same holds for $\amax$ even though it is not a cyclic pair set.
\begin{corollary}\label{c:max-r-amax}
  Suppose that $A$ is a partial cyclic pair set and $x$ is a maximal
  point of $\amax$. Then the point $x-r$ is not an $r$-maximal point
  of $\amax$.
\end{corollary}
\begin{proof}
  Again, by Lemma~\ref{l:max-r}, $x-r$ would be an $(r+m)$-maximal
  point of $\amax$ so an $m$-maximal point of $\amax$. However, by a
  similar argument to Lemma~\ref{l:unique-max}, $\amax$ has a unique
  maximal point (the proof of that lemma just uses the fact that the
  set has exactly one point of some pair).
\end{proof}
\begin{lemma}\label{l:max-min}
  Suppose that $A$ is a (full) pair set, that $0$ is maximal, that $x$
  is $r$-minimal for some $x$ and $r$, and that $y$ is any $x$-maximal
  point. Then $y+x$ is $r$-minimal.
\end{lemma}
\begin{proof}
  Since $0$ and $y$ are both $x$-maximal we have that
  $A|_{[0,x)}=A|_{[y,y+x)}$ and, since $0$ is maximal, we have
  $A|_{[0,x+r)}\ge A|_{[y,y+x+r)}$. Combining these we get $A|_{[x,x+r)}\ge
  A|_{[y+x,y+x+r)}$. But since $x$ is $r$-minimal, we have
  $A|_{[x,x+r)}\le A|_{[y+x,y+x+r)}$ and hence
  $A|_{[x,x+r)}=A|_{[y+x,y+x+r)}$. Thus, $y+x$ is $r$-minimal.
\end{proof}
\begin{lemma}\label{l:not-min}
  Suppose that $A$ is a (full) pair set and that $0$ is maximal. Then, for
  all $r<2m'$, no point in $[0,r]$ is $r$-minimal. In particular, this
  holds for $r=m'$.
\end{lemma}
\begin{proof}
  Suppose that $x\in [0,r]$ is $r$-minimal. Since $r\ge x$
  Lemma~\ref{l:max-min} implies that, for any $y$ that is $x$-maximal,
  $y+x$ is $x$-minimal, and this is the form we shall use.

  Now $x$ is $r$-minimal, so $x+2m'$ is $r$-maximal and thus, since
  $x\le r$, also $x$-maximal. By Lemma~\ref{l:max-min}, $2x+2m'$ is
  $x$-minimal, and hence $2x$ is $x$-maximal. By Lemma~\ref{l:max-min}
  again, we have that $3x$ is $x$-minimal, so $3x+2m'$ is
  $x$-maximal. Repeating we see that $4x+2m'$ is $x$-minimal so $4x$
  is $x$-maximal etc. In particular $kx$ is $x$-maximal for even $k$
  and $x$-minimal for odd $k$.  Thus the pair set $A$ is periodic with
  period $2x\le2r<m$. Since a cyclic pair set cannot be periodic this
  is a contradiction.
\end{proof}
\begin{lemma}\label{l:not-top}
  Suppose that $A$ is (full) pair set. Suppose that $x$ is an $m'$-maximal
  point. Then the maximal point lies in $(x-2m',x]$.
\end{lemma}
\begin{proof}
  Let $y$ be the maximal point. 

  First, suppose that $y\in [x+1,x+m']$. Then, by
  Corollary~\ref{c:max-r}, $x$ is not $m'$-maximal, which is a
  contradiction.

  Now suppose that $y\in [x+m',x+2m']$. Then $x+2m'$ is $m'$-minimal
  but, since $x+2m'\in [y,y+m')$, this contradicts Lemma~\ref{l:not-min}.
\end{proof}

\begin{corollary}\label{c:least-max-point}
  Suppose that $A$ is a full pair set and that $0$ is $m'$-maximal. Then
  the actual maximal point is the first point in the set
  $(2m',4m']$ that is $m'$-maximal.
\end{corollary}
\begin{proof}
  Let $x$ be the first point in $(2m',4m']$ that is $m'$-maximal.  By
  Lemma~\ref{l:not-top} applied to $0$, we know that the actual
  maximum is in $(2m',4m']$ and, by applying it to $x$, that the
  actual maximum is not in $(x-2m',x]$. Combining these we see that
  the actual maximum must be in $(2m', x]$. Since the actual maximum
  must be $m'$-maximal this implies that $x$ is the actual maximum.
\end{proof}

\begin{lemma}\label{l:earliest-latest}
  Suppose that $A$ is a partial pair set and that $x$ is $m'$-maximal
  in $\amax$ and that $y$ is such that there are no free points in
  the interval $[x,y)$. Let $A'=A+\one|_{[y,y+m')}+\one|_{[y-m',y)}$ and let
  $x'$ be the maximum point of $A'$.

  Then
  \begin{enumerate}[label=(\alph*)]
    \item\label{t:p:part-0}
    $x'\in(x-2m',x]$,
    \item\label{t:p:part-1}
    $x'$ is $m'$-maximal in $\amax$,
    \item\label{t:p:part-2} if $x'\not \in (y-m',x]$ then $A$ has no
    free points in the interval $[x',y-m')$,
  \item\label{t:p:part-3} for any (full) pair set $B$ extending
  $A+\one|_{[y,y+m')}$, the maximum point of $B$ lies in the
  interval $[x',x]$.
\end{enumerate}
\end{lemma}
\begin{proof}
  First observe that, since there are no free points in $[x,y)$ we
  have $A'|_{[x,x+m')}=\amax|_{[x,x+m')}$. This observation, together
  with $x'$ being $m'$-maximal in $A'$, shows that
  \[A_{\max}|_{[x',x'+m')}\ge A'|_{[x',x'+m')}\ge
  A'|_{[x,x+m')}=\amax|_{[x,x+m')}.\] Since $x$ is $m'$-maximal in
  $\amax$ this shows that $x'$ is also $m'$-maximal in $\amax$ and
  that each of the inequalities must actually be an equality.

  This shows that $x$ is $m'$-maximal in $A'$, so by
  Corollary~\ref{c:least-max-point}, $x'\in (x-2m',x]$ which is
  part~\ref{t:p:part-0}.  It also shows that $x'$ is $m'$-maximal in
  $\amax$ which is part~\ref{t:p:part-1}.

  Moreover, it shows that $A_{\max}|_{[x',x'+m')}= A'|_{[x',x'+m')}$
  so, in particular, if $x'\not \in (y-m',x]$ then, since any free
  points in the interval $[x',y-m')$ would be present in $\amax$ but
  absent in $A'$ there cannot be any such free points.

  To prove part~\ref{t:p:part-3} observe that, since there are no free
  points in $[x,y)$, $B|_{[x,y)}=A'|_{[x,y)}$, and by definition
  $B|_{[y,y+m')}=A'|_{[y,y+m')}$. Thus, 
  \begin{equation}
  B|_{[x,y+m')}=A'|_{[x,y+m')}.\tag{$*$}\label{e:x,y+m'}    
  \end{equation}

  In particular $B|_{[x,x+m')}=A'|_{[x,x+m')}=\amax|_{[x,x+m')}$, so $x$ is
  $m'$-maximal in $B$. Let $z$ be the maximum point of $B$.  By
  Corollary~\ref{c:least-max-point}, we see that $z\in (x-2m',x]$. 

  Also, since both $A'$ and $B$ are full pair sets, (\ref{e:x,y+m'})
  shows that $B|_{[x-2m',y-m')}=A'|_{[x-2m',y-m')}$.  By definition
  $A'|_{[y-m',y)}\ge B|_{[y-m',y)}$ (i.e. restricted to this interval,
  $A'$ contains $B$). Combining these, we see that
  $A'|_{[x-2m',y)}\ge B|_{x-2m',y)}$.  In particular, any $m'$-maximal
  point of $B$ in $[x-2m',x]$ is also $m'$-maximal in $A'$. Thus by
  Corollary~\ref{c:least-max-point} we see that $x'\le z$.
\end{proof}

We are now in a position to prove Lemma~\ref{l:key-lemma}.
\begin{proof}[Proof of Lemma~\ref{l:key-lemma}]\mbox{}

  We may assume that 0 is the maximum point of $\amax$.  Write $A^k$
  for the set $A+\one|_{[km',(k+1)m')}$ (so $A^4=A^0$ etc -- we use
  whichever expression is convenient). Similarly, write $A^{k,k+1}$ for
  the set $A+\one|_{[km',(k+1)m')}+\one|_{[(k+1)m',(k+2)m')}$ and note
  that $A^{k,k+1}$ is a cyclic pair set.

  Let $x_k$ be the maximum point of $A^{k,k+1}$.  Observe that
  $A^{4,5}|_{[0,2m')}=\amax|_{[0,2m')}$ so 0 is maximal in $A^{4,5}$;
  i.e., $x_4=0$.

  By Lemma~\ref{l:earliest-latest}\ref{t:p:part-1} applied with
  $x=y=0$, the maximum point $x_3$ of $A^{3,4}$ is $m'$-maximal in
  $\amax$. By Corollary~\ref{c:max-r-amax} either $x_3=x_4$ or
  $x_3<3m'$. If $x_3=x_4$ then, by
  Lemma~\ref{l:earliest-latest}\ref{t:p:part-3}, we see that the
  maximum point of any set $B$ extending $A^0$ is $0$, and the claimed
  result is trivially true with $z_1=z_2=0$.

  Thus, we may assume $x_3<3m'$. Then
  Lemma~\ref{l:earliest-latest}\ref{t:p:part-2} shows that there are
  no free points in $[x_3,3m')$.  Now we can apply
  Lemma~\ref{l:earliest-latest}\ref{t:p:part-1} again but this time with $x=x_3$ and
  $y=3m'$. This shows that $x_2$ is $m'$-maximal in $\amax$.

  If $x_2> 2m'$ then, by
  Lemma~\ref{l:earliest-latest}\ref{t:p:part-3}, any set extending
  $A^3$ has maximum point in $[x_2,x_3]$ and any set extending $A^4$
  has maximum point in $[x_3,0]$ (recall $x_4=0$). Thus, in this case
  we are done with $t=x_3$ and $s=x_3-x_2$.

  Thus we may assume $x_2\le 2m'$ and thus, by
  Lemma~\ref{l:earliest-latest}\ref{t:p:part-2}, that that there are
  no free points in $[x_2,2m')$.  We apply
  Lemma~\ref{l:earliest-latest}\ref{t:p:part-1} again, this time with $x=x_2$ and
  $y=2m'$. This shows that $x_1$ is $m'$-maximal in $\amax$.

  If $x_3-x_1<2m'$ then, as above, we are done: any set extending $A^2$
  has maximum point in $[x_1,x_2]$ and any set extending $A^3$ has
  maximum point in $[x_2,x_3]$. 

  Thus we may assume $x_3 -x_1\ge 2m'$ and, in particular, that
  $x_1\le m'$. Hence, Lemma~\ref{l:earliest-latest}\ref{t:p:part-2}
  implies that $A$ has no free points in $[x_1,m')$.

  To summarise, we have that, for each $k$, $x_k$ is $m'$-maximal
  in $\amax$, $x_k\le km'$, and $A$ does not have any free points
  in the interval $[x_k,km')$.
  
  If, for any $k$ we have $x_{k+2}-x_k<2m'$ then as above we are done:
  any set extending $A^{k+1}$ has maximum point in $[x_k,x_{k+1}]$ and any set
  extending $A^{k+2}$ has maximum point in $[x_{k+1},x_{k+2}]$.

  Hence the only remaining case is that both $x_2=2m'$ (i.e
  $x_2=x_0+2m'$) and $x_3=x_1+2m'$.  Then both $0$ and $2m'$ are
  $m'$-maximal in $\amax$ and, in particular,
  $A|_{[0,m')}=A|_{[2m',3m')}$. This means that there can be no fixed
  point (i.e., non-free point) in these intervals as such a point
  would be fixed differently in the two sets. Hence all points in
  these intervals are free.

  But, if this is the case, then $0$ is the maximal point
  in any set $B$ extending $A^0$. Indeed, since $0$ is maximal in
  $\amax$ we must have that the preceding element (element `$-1$') is
  fixed not in $A$. Thus, in the interval $[2m',4m')$, the set $B$
  consists of $m'$ zeros, $m'-1$ points we don't know about followed
  by another zero. Thus it is trivial to see that, whatever $B$ is, no
  point in $[2m',4m')$ is $m'$-maximal (it would have to have $m'$
  ones following it) and the result follows.
\end{proof}

\section{Some Simple Observations}\label{s:simple}
Before proceeding further we collect some simple observations. Note
that, although some of the proofs are a little long to write out, all
the results in this section are essentially trivial consequences of
our work so far.

\subsection{Bounded size lines}
In Theorem~\ref{t:size2=p2-win} we showed that all transitive games with a line of size 2 are a PII
win. However, this is best possible since our first example of an even
sized \PI win was a game on 6 points with all lines of size 3. It is
easy to extend this to find arbitrarily large examples of \PI win
games with all lines of size 3.
\begin{theorem}\label{t:3-uniform}
  For all $n=12k+6$ there is a transitive avoidance game on
  $n$ points, with all lines of size 3, that is a first player win.
\end{theorem}
\begin{proof}
  Let $\cH_0$ be the transitive avoidance game on 6 points that is a
  \PI win given by Proposition~\ref{p:2b} and let $\Phi$ be a winning
  strategy for $\cH_0$. Let $\cH$ be the disjoint union of  $2k+1$
  copies of $\cH_0$ which we write as $[2k+1]\times \cH_0$.

  Obviously $\cH$ has $12k+6$ points, all lines have size 3, and it is
  transitive.  Thus we just need to show that this game is a \PI win.

  Let $f$ be an involution from $[2k+1]\to[2k+1]$ fixing $1$ and
  having no other fixed point. \PI starts by playing according to the
  winning strategy in the first copy of $\cH_0$. For all subsequent
  moves he plays as follows. Suppose that \PII has just played a point
  $(x,y)\in [2k+1]\times \cH_0$. If $x\not=1$ then \PI plays
  $(f(x),y)$; if $x=1$ then \PI plays according to $\Phi$ in the first
  copy of $\cH_0$. It is easy to see that \PI can follow this strategy
  (i.e., he never has to play a point that has already been played), and
  that this strategy is a \PI win.
\end{proof}
One might wonder whether size 3 is special in the above theorem, but it
is easy to show that there are arbitrarily large games that are \PI wins
and have all lines of any fixed size greater than 2.

\begin{corollary}\label{c:r-uniform}
  For any $r\ge 3$ and $n_0$ there is a transitive game on $n>n_0$ points
  with lines all of size $r$ that is a first player win.
\end{corollary}
\begin{proof}
  Take the game $\cH$ with lines $\cL$ given by 
  Theorem~\ref{t:3-uniform} with $k=\max(n_0,2r)$. Define the new game
  $\cH'$ to have the same board and to have lines 
  \[\cL' =\{L'\in [n]^{(r)}:L'\supset L \text{ for some $L\in\cL$}\}.\]

  Let \PI play as above. Then at some point \PII forms a set
  $L\in\cL$. If \PII has played at least $r$ points then he has formed
  a set $L'\in\cL'$ (any superset of a set $L$ which has size $r$ is
  in $\cL'$). If \PII has not played this many points then \PI just
  has to continue playing until \PII has played $r$ points without
  forming a set in $\cL'$ himself. This is trivial: on each turn \PI
  plays in a copy of $\cH_0$ that has not been played in previously
  (where $\cH_0$ is as in the construction for $\cH$). This is
  possible as $k\ge 2r$.
\end{proof}
\subsection{Fast Player I wins}
In the examples of \PI wins we have given so far \PII can avoid losing
for a long time: indeed, in the even case he need not lose until the
last turn of the game, and in the odd case he can play for at least time
$n/4$ (in the construction in Theorem~\ref{t:n-composite-odd} for
board size $pq$ all lines have size $p'q'\ge pq/4$).

However, an iterated variant of the odd size `majority of
majorities' construction in Theorem~\ref{t:n-composite-odd} shows that
there are games where \PII must lose in time $o(n)$.
\begin{theorem}
  For any $\eps>0$ there is a game $\cH$ on $n$ points such that \PII
  must lose before time $\eps n$.
\end{theorem}

The proof of this Theorem, whilst simple, is a little tedious to write
out. Since we prove a substantially stronger result in the next
section we omit the proof.

\subsection{Transitivity}
So far we have required that the game be transitive on the
points. However, many natural games are more transitive than this. For
example the Ramsey Avoidance Game is also line transitive
(recall that edges in that game correspond to our points, and complete
graphs there are our lines). In fact, the Ramsey Avoidance Game is
point/line transitive in that any line, and any point in that line can
be mapped, by an element of the automorphism group, to any other line
and point in that line.

All our \PI wining games so far are not transitive on the lines so it
is natural to ask whether this extra transitivity is enough to rule
out the possibility of a \PI win. We answer this in the next section.

We remark that our `majority of majorities' game may appear transitive
on the lines, and it is true that the it is transitive on the set
$\cW$ of `allowed subsets'. However, it is \emph{not} transitive on
the set $\cL$ of lines.
\section{The Torus Game}\label{s:torus}
In this section we introduce a very natural game that simultaneously
satisfies all of the properties we discussed in the previous section:
\PI can force a win in time $o(n)$, all lines have size 3, and it is
point/line transitive (indeed it is even more: it is also possible to
map any point and line not containing it to any other point and line
not containing it).  
\begin{defn}
  The torus game $\cT_q(d)$ is the game on board $\Z_q^d$, with lines
  $\cL$ defined to be all subsets of the form $\{x,x+y,x+2y,\dots
  x+(q-1)y\}$ for $x,y\in \Z_q^d$ and $y\not=0$.
\end{defn}

We show that the game $\cT_3(d)$ is an avoidance game satisfying all
the conditions mentioned at the start of this section. Obviously, all
lines have size 3, and it is transitive (translation by any element of
$\Z_3^d$ is in its automorphism group). The other stronger
transitive properties mentioned above are also simple to verify.

\begin{theorem}
  The game $\cT_3(d)$ is not a \PII win.
\end{theorem}
\begin{proof}
  The map $g:\Z_3^d\to \Z_3^d $ defined by $g(x)=-x$ is an involution of the board $X$
  with a single fixed point $0$. \PI plays $0$ on his first go. Then
  on each subsequent turn he plays $g(y)$ where $y$ is the point \PII
  just played.

  Observe that this is a valid strategy: \PI never plays a point that
  has already been played. 

  Now suppose that at some point \PI forms a set $L\in \cL$. This set
  cannot contain $0$ since all lines containing $0$ are of the form
  $\{-y,0,y\}$, and if \PI has played $y$ then \PII has played
  $-y$. Thus, if \PI has all the points of $L$ then \PII has already
  played all points of $-L$ which is also a line in $\cL$. Thus \PII
  has already lost.
\end{proof}

\begin{corollary}
  For all sufficiently large $d$ the game $\cT_3(d)$ is a \PI
  win. Moreover, if $t_d$ is the longest \PII can avoid losing then
  $t_d/3^d\to 0$ as $d\to\infty$.
\end{corollary}
\begin{proof}
  View the torus $\Z_3^d$ as a Hales-Jewett cube $[3]^d$. Observe that
  any combinatorial line in the Hales-Jewett cube is a line in the
  game sense (i.e. is in $\losing$). Now for any $d$ greater than the
  Hales-Jewett number $HJ(q,2)$ (i.e., the smallest $d$ such that any
  two colouring of the cube $[3]^d$ contains a monochromatic
  combinatorial line) the game cannot be a draw, and thus must be a
  \PI win.
  
  To prove the bound on the time observe that, by the density version
  of the Hales-Jewett Theorem~\cite{MR1191743} (or, indeed, standard
  cap-set results), there exists a sequence $\eps_d$ tending to zero
  such that any set of size $\eps_d 3^d$ in $[3]^d$ contains a
  combinatorial line. Thus, by time $\eps_d 3^d$ one player must have
  lost, so \PII must have lost. Hence $t_d/3^d<\eps_d$ so $t_d/3^d\to
  0$ as claimed.
\end{proof}

We conclude with an example showing that there are also even-sized
boards where \PI can win quickly. Moreover, in this example all
lines have size $3$.
\begin{theorem}
  There are games $\cH_d$ with board size $n_d=6\cdot 3^d$ which are
  \PI wins and, moreover, \PII loses game $\cH_d$ in time $o(n_d)$.
\end{theorem}
\begin{proof}
  The construction is an extension of the above. Let $\cH_1$ be the
  game on six points that is a \PI win given by
  Proposition~\ref{p:2b}. We define $\cH_d$ to be $\cT_3(d)\times\cH_1$ where the lines in
  this product are a line in one of the component directions and
  constant in the other.  More precisely, the board is the set
  $\Z_q^d\times \cH_1$ and the lines are the set
  \[
  \left(\bigcup_{\substack{ x\in \cT_3(d) \\ L\in \cL(\cH_1)}}\{x\}\times L\right)\cup
  \left(\bigcup_{\substack{ L\in \cL(\cT_3(d)) \\ y\in \cH_1}} L \times\{y\}\right).
  \]

  We have to show that \PI wins this game and that the game ends
  quickly.

  First, we give a winning \PI strategy. Let $\Phi$ denote the winning
  strategy for \PI in $\cH_1$. As usual we view the board of $\cT_3(d)$ as
  $\{-1,0,1\}^d$. \PI starts by playing in the $(0,0,\dots,0)$ copy of
  $\cH_1$ and plays according to $\Phi$. We call this copy of $\cH_1$ the
  zero copy.

  Now for all subsequent moves \PI does the following. If \PII just
  played in the zero copy the \PI also plays in the zero copy and
  follows the strategy $\Phi$. If \PII played in any other copy
  $(x_1,x_2,\dots,x_d)$ of $\cH_1$ then \PI plays the same point in
  the antipodal copy $(-x_1,-x_2,\dots,-x_d)$ of $\cH_1$.

  Observe that \PI cannot lose by forming any triple not including a
  point in the zero copy of $\cH_1$ as \PII would have formed the
  antipodal set and would have already lost. Furthermore \PI cannot lose
  by forming a triple not wholly contained in the zero copy as \PII would
  have the antipodal point. Finally, \PI does not form a losing triple
  in $\cH_1$ as he is following the winning strategy $\Phi$ there.

  To show that the game ends quickly consider the points in
  $\cT_3(d)\times\{y\}$ for some $y\in\cH_1$. If a player has more
  than $\eps \|\cT^d\|$ points in this set then, by the density
  Hales-Jewett theorem, provided $d$ is sufficiently large, he must
  have a combinatorial line which is a losing set in our game. Thus,
  if \PII has not lost he has played a total of at most
  $6\eps|\cT^d|=\eps |\cH_d|$ points.
\end{proof}

\section{Open Problems}
Our first open question concerns the case when $n$ is prime.
\begin{question}
  For which primes $n$ does there exist a transitive avoidance game on
  $n$ points that is a first player win?
\end{question}
We know almost nothing in this case. When $n=3$, $5$ or $7$ there is
no transitive avoidance game that is a \PI win.  Indeed, $3$ is
trivial; $5$ follows immediately from Theorem~\ref{t:graph}. The case
$|V|=7$ is slightly trickier but, by Theorem~\ref{t:graph}, we only
need to consider the case where all lines have size three and, since $7$ is prime, we may
assume that the cyclic group $C_7$ acts on the board. This
reduces the problem to a manageable number of cases.

However, for 11 and 13 there \emph{are} transitive avoidance games
that are first player wins. These were found by computer search. The
games we find are of the following form: there is a transitive
intersecting family $\cW$ of $\frac{n-1}{2}$-sets that \PI can
guarantee to make one of in his first $\frac{n-1}{2}$ moves. Thus
setting $\cL$ to be all other $\frac{n-1}2$-subsets we get the
required game.

For $n=11$, $\cW$ is the set of all affine copies of the set
$\{0,1,2,4,5\}$. For $n=13$, $\cW$ is the set of all affine copies of
the any of the sets $\{0,1,2,4,5,6\}$, $\{0,1,2,4,5,7\}$ or
$\{0,1,3,4,5,7\}$. However, we do not have a `nice' strategy for
either game.

We remark that we searched for games transitive under the affine group
as this reduced the number of orbits to a level manageable by computer
search. There are other \PI wins which are less symmetric, and there
are some which are not of this `maker' form. 

We do not know the anything about the answer to this question for any
primes greater than~13. Indeed, we do not even know if there are
infinitely many such primes. Moreover, the number of possible
transitive games for 17 is extremely large (even if we restrict to
affinely transitive games) which makes a computer search, even for the
next open case, impractical.

\medskip
Theorem~\ref{t:3-uniform} and Corollary~\ref{c:r-uniform} answer the
question of for which $n$ and $r$ there exist transitive avoidance games on $n$
points with lines all of size $r$ that are \PI wins, for infinitely many values of $n$ and
$r$. However, the full characterisation remains open.

\begin{question}
  For which $n$ and $r$ does there exist a transitive
  avoidance game on $n$ points with all lines of size $r$ that is a \PI win.
\end{question}
In particular we do not know the full characterisation even for $r=3$.

\medskip
A family of games of particular interest is the class of
\emph{sim-like} games, which is defined as follows. These games have
board the edge set of the complete graph $K_n$.  The lines are sets of
edges that are isomorphic to some forbidden graph (or family of
graphs). For example, the Ramsey Avoidance Game $RAG(n,k)$ is of this
form: the board is the edge set of $K_n$ and the lines are the
subgraphs isomorphic to $K_k$. (They have been called \emph{sim-like}
as the first non-trivial Ramsey Avoidance Game, $RAG(6,3)$, is
commonly called Sim.)

We do not know if there are any sim-like games that are \PI wins.
\begin{question}
  Does there exist a sim-like game that is a \PI win?
\end{question}

One particular property of sim-like game is that they have a large
automorphism group. Indeed, the automorphism group of a sim-like game
played on $K_n$ trivially contains $S_n$. We do not know whether this
itself is enough to force a \PII win.
More generally it would be interesting to characterise for which
automorphism groups there exist \PI wins.

\begin{question}
  For which groups $G$ does there exist a transitive avoidance game 
  with automorphism group $G$ that is a \PI win?
\end{question}

\medskip Finally, there is a natural definition of an infinite
avoidance game. In this case the board has infinite size (with all
lines finite) and a player loses if he forms a line. If the play
continues forever with neither player losing then the game is deemed a
draw.
\begin{question}
  Is there an infinite transitive avoidance game that is a \PI win?
\end{question}
We have no intuition as to the correct answer to this question;
indeed, we do not even know the answer when all losing lines have size
three. We remark that any such game must have `many' losing lines, to
avoid \PII having an easy draw -- for example, by picking a faraway
point that does not complete a losing line.  
\bibliography{mybib}{}
\bibliographystyle{abbrv}

\end{document}